%% file: main.tex
\documentclass[letterpaper,11pt]{amsart}

\usepackage[utf8]{inputenc}
\usepackage[english]{babel}
\usepackage[margin = 1.25in]{geometry} 
\usepackage{amsthm}
\usepackage{amsmath}
\usepackage{amssymb}
\usepackage[scr]{rsfso}
\usepackage{csquotes}
\usepackage{bm,bbm}
\usepackage{color}
\usepackage[dvipsnames]{xcolor}
\usepackage{graphicx}
\usepackage{caption, subcaption}
\usepackage{array}
\usepackage{tikz}
\usetikzlibrary{matrix,arrows}
\usepackage{tikz-cd}
\usepackage[framemethod=TikZ]{mdframed}
\usepackage{youngtab}
\usepackage{mathtools}
\usepackage[normalem]{ulem}
\usepackage[all]{xy}
\usepackage{rotating}
\usepackage{extarrows}
\usepackage{accents}
\usepackage{leftidx}
\usepackage{upgreek}
\usepackage{xr-hyper}
\usepackage[draft]{todonotes}
\usepackage{proofread}
\usepackage{hyperref}
\usepackage{macros}

\hypersetup{  
  pdftitle    = {A note on semiorthogonal indecomposability of some Cohen-Macaulay varieties},
  pdfauthor   = {Dylan Spence},
  colorlinks=true,
  citecolor=orange,
  linkcolor=blue}

\numberwithin{equation}{section}
\title[A note on semiorthogonal indecomposability of some Cohen-Macaulay varieties]{A note on semiorthogonal indecomposability of some Cohen-Macaulay varieties}
\begin{document}
\input{abstract}

\input{authors}
\maketitle
\setcounter{secnumdepth}{1}
\setcounter{tocdepth}{1}

\input{intro}
\input{prereq}
\input{sods}

\bibliographystyle{alpha} 
\bibliography{references}

\end{document}

%% file: abstract.tex
\begin{abstract}
    In this short note, we observe that the criterion proven in \cite{KawataniOkawa-Nonexistence} for semiorthogonal indecomposability of the derived category of smooth DM stacks based on the canonical bundle can be extended to the case of projective varieties with Cohen-Macaulay singularities. As a consequence, all projective curves of positive arithmetic genus have weakly indecomposable bounded derived categories and indecomposable categories of perfect complexes. 
\end{abstract}


%% file: authors.tex
\author{Dylan Spence}

\address{
  Department of Mathematics\\
  Indiana University\\
  Rawles Hall\\
  831 East 3rd Street\\
  Bloomington, IN 47405\\
  USA
}
\email{dkspence@indiana.edu}

%% file: intro.tex
\section{Introduction}


The bounded derived category of coherent sheaves $\D^b(X)$ on a projective variety has received a large amount of attention in the last 30 years or so. As in many areas of mathematics, it is often convenient and enlightening to ``decompose" the derived category into pieces. To be precise, for a projective variety $X$, we say that the derived category $\D^b(X)$ admits a semiorthogonal decomposition into two full triangulated subcategories $\A$ and $\B$ if there are no nonzero morphisms from objects of $\B$ to objects of $\A$, and the smallest full triangulated subcategory containing $\A$ and $\B$ is all of $\D^b(X)$. Of course, this definition also makes sense for any triangulated category $\T$ in place of $\D^b(X)$. If any semiorthogonal decomposition of a triangulated category $\T$ is trivial, that is, one of the components is zero, then we say that $\T$ is indecomposable.

In the geometric setting, semiorthogonal decompositions are more then just convenient tools for the study of $\D^b(X)$ and $\perf X$. For example there are a number of interesting conjectures relating the existence of certain semiorthogonal decompositions to the rationality of $X$ (see \cite{Kuznetsov-rationality} for a somewhat recent survey). Despite the interest however, semiorthogonal decompositions can be hard to produce without inspiration, however an interesting special case is when the derived category is indecomposable. This occurs, for example, for smooth varieties $X$ with $\omega_X \cong \O_X$, or for smooth curves of positive genus \cite{Okawa-semiorthogonal_decomposability_of_curves}.  

The most general indecomposability criteria to date was proven in \cite[Theorem 3.1]{KawataniOkawa-Nonexistence}. There, the authors show that any smooth projective variety (more generally, any smooth and proper DM stack with a non-stacky point) with small enough canonical base locus has an indecomposable derived category. In this note we observe that these techniques generalize, with suitable assumptions, to the case of Cohen-Macaulay projective varieties. 

To be more specific, let $X$ be a Cohen-Macaulay projective variety and $x\in X$ any closed point. The key notion in the proof is that of a \emph{Koszul zero-cycle}, which is a closed subscheme $Z_x$ of $X$ whose underlying topological space is just the closed point $\{x\}$, but whose structure sheaf $\O_{Z_x}$ is perfect (that is, a bounded complex of locally free sheaves of finite rank) as an object of $\D^b(X)$, and whose support is exactly $\{x\}$. These skycraper sheaves form a spanning class for both $\perf X$ and $\D^b(X)$, and are a replacement for the sheaves $k(x)$, which are not perfect if $x\in X$ is a singular point.

Our main result is an application of these objects, and generalizes \cite[Theorem 3.1]{KawataniOkawa-Nonexistence}.

\begin{theorem}\label{Main Theorem}
    Let $X$ be a Cohen-Macaulay projective variety with dualizing sheaf $\omega_X$. Suppose that $\perf X = \langle \A, \B \rangle$ is a semiorthogonal decomposition. Then:
  \begin{enumerate}
    \item for any closed point $x \in X \setminus \operatorname{Bs}|\omega_X|$, and all Kozsul zero-cycles $Z_x$ at $x$, exactly one of the following holds:
    \begin{enumerate}
      \item $\O_{Z_x} \in \A$, or 
      \item $\O_{Z_x} \in \B$.
    \end{enumerate}
    \item When (1a) (resp. (1b)) is satisfied, then the support of any object in $\B$ (resp. $\A$) is contained in $\operatorname{Bs}|\omega_X|$.
  \end{enumerate}
\end{theorem}

Along the way, we make sense of what the base locus of $\omega_X$ is, given that in the Cohen-Macaulay setting, it is not a line bundle. Combining this result with some auxiliary statements regarding the support of semiorthogonal components, we are then able to apply this to curves, yielding the following.

\begin{corollary}
  Let $C$ be a projective curve of arithmetic genus at least one. Then $\perf C$ is indecomposable and $\D^b(C)$ is weakly indecomposable.
\end{corollary}

\noindent Here weakly indecomposable refers to the nonexistence of semiorthogonal decompositions where the components are admissible (see Definition \ref{def_admissible}).

\subsection{Acknowledgements}
The author wishes to thank Valery Lunts for his support, encouragement, and advice. The author also wishes to thank A. Thomas Yerger for helpful discussions. 

\subsection{Conventions}
Throughout we work over a fixed algebraically closed field $k$. In addition, unless explicitly stated otherwise, any scheme or variety $X$ is assumed to be a projective variety; in particular for us this means an integral separated scheme of finite type over $\spec k$, such that $X$ admits a closed immersion into some projective space. A curve refers to a projective variety of dimension one. 

By $\D^*(X)$ we always mean the derived category of complexes of coherent sheaves respectively, with appropriate decoration $* \in \{\phantom{0},-,+,b\}$ meaning unbounded, bounded above, below, and bounded respectively. In particular, we identify $\D^b(X)$ with the full subcategory of $\D(\qcoh X)$ consisting of complexes of quasi-coherent sheaves with bounded coherent cohomology sheaves. By $\perf X$ we mean the full triangulated subcategory of $\D(\qcoh X)$ consisting of complexes which are isomorphic (in $\D(\qcoh X)$) to a bounded complex of vector bundles. 

%% file: prereq.tex
\section{Preliminaries}

\subsection{Base loci for coherent sheaves}

To formulate our results, we need to define what we mean by the base locus of a sheaf. Let $X = \cup_j Y_j$ be a reduced quasi-projective scheme of finite type over $k$ (that is, $X$ is the finite union of the quasi-projective varieties $Y_j$), $\mathcal{F}$ a coherent sheaf on $X$, and $s \in H^0(X,\mathcal{F})$ a global section. We define the set $$D(s) = \{x\in X \text{ closed}\, | \, 0 \neq s(x) \in \mathcal{F}(x) \},$$ where $\mathcal{F}(x)$ is the fiber of $\mathcal{F}$. We denote the complement as $V(s) = D(s)^c$, which we refer to as the zero set of $s$.

Suppose for the sake of illustration that $X$ was integral (so that by our conventions it is a quasi-projective variety), and that $\mathcal{F}$ is locally free. Then the definition of $V(s)$ and $D(s)$ obviously agree with the corresponding definitions for the zero set (respectively, its complement) of a section of a vector bundle. In this setting, so long as the section is nonzero, the sets are closed (respectively, open), but if $\F$ is not a vector bundle, this is no longer true. Instead, these sets are constructible.

Recall that a locally closed subset of $X$ is the intersection of an open and closed subset. A set is said to be \emph{constructible in $X$} if it is a finite union of locally closed subsets of $X$. Constructible sets are closed under finite unions, finite intersections, and complements.

\begin{lemma}\label{lem_constructible}
  Let $X = \cup_j Y_j$ be a reduced quasi-projective scheme of finite type over $k$, $\mathcal{F}$ a coherent sheaf on $X$, and $s \in H^0(X,\mathcal{F})$ a global section. Then $D(s)$ is a constructible subset of $X$.
\end{lemma}
\begin{proof}
  We proceed by induction on the dimension of $X$. Clearly, if $X$ is a finite number of points, then the set $D(s)$ is a finite disjoint union of points (possibly empty), and so is constructible.

  Now if $\dim X \geq 0$, let $Y \subset X$ be an irreducible component of positive dimension. Then via pullback along $Y \hookrightarrow X$, we assume for now that $X$ is integral and $\mathcal{F}$ is a coherent sheaf on $X$. By well-known facts, there is an open dense subset $j:U\hookrightarrow X$ such that $\mathcal{F}$ is locally free of non-negative rank. Denoting $s|_U := j^* s$, we see $D(s|_U) = D(s) \cap U$ is constructible in $X$. Indeed, $D(s|_U)$ is open in $U$, since $j^*\mathcal{F}$ is locally free, and thus $U \setminus D(s|_U)$ is closed in $U$, and so is of the form $V \cap U$ for some closed subset $V$ of $X$, and so $(X \setminus V) \cap U = D(s|_U)$ is an open subset of $X$, hence constructible. 
  
  Now, still assuming $X$ is integral, the closed subset $X \setminus U = Z$ is a finite union of subvarieties (after possibly giving them their natural reduced subscheme structure) of lower dimension.  Let $i:Z \hookrightarrow X$ be the closed immersion and set $s|_{Z} : = i^* s$ as before. Then by the induction hypothesis, the set $D(s|_{Z}) = D(s) \cap Z$ is constructible and we see that $D(s) = D(s) \cap U \coprod D(s) \cap Z$ is constructible. 

  Now if $X = \cup_j Y_j$ is not integral with finitely many irreducible components $Y_j$, $D(s)$ is the union of its intersections with each $Y_j$, which we have shown above are each constructible. Using that constructible sets are closed under finite unions we have that $D(s)$ is constructible.
\end{proof}

\begin{remark}\label{rem_denseopen}
  Let $X$ be as above, $\mathcal{F}$ a torsion-free coherent sheaf on $X$, $s$ a global section of $\mathcal{F}$, and $Y \subset X$ is an irreducible component where $s|_Y$ is nonzero (so that $D(s|_Y)$ is not empty), then $D(s)$ is dense in $Y$. This follows as on $Y$, there is a dense open subset $U\subset Y$ where $\mathcal{F}$ is locally free, and by elementary topology, any dense subset of $U$ will also be dense in $Y$. Since $D(s|_U)$ is dense in $U$, this proves the claim. Further, since $D(s|_Y)$ is constructible and dense for $s$ nonzero, it will contain a dense open subset of $Y$, as given any topological space, a constructible subset always contains a dense open subset of its closure (see for example \cite[Lemma 2.1]{An-Rigid_geometric}). 
\end{remark}

With the above in mind, we make the following definition.

\begin{definition}
  Let $X$ be as in Lemma \ref{lem_constructible}. Given a coherent sheaf $\mathcal{F}$ on $X$, we define the base locus of $\mathcal{F}$ to be $$\operatorname{Bs}|\mathcal{F}| = \bigcap_{s\in H^0(X,\mathcal{F})} \overline{V(s)}.$$ The sheaf $\mathcal{F}$ is then base point free if this intersection is empty, or equivalently, if the interiors of the sets $D(s)$, ranging over $H^0(X,\mathcal{F})$, are an open cover of $X$.
\end{definition}

\noindent Note that by definition, the base locus of a sheaf is always a closed subset of $X$.

In the case of curves however, things simplify nicely.

\begin{lemma}
  Let $\F$ be a torsion-free sheaf on an curve $C$. Then $D(s)$ is open for any global section $s\in H^0(X,\mathcal{F})$.
\end{lemma}
\begin{proof}
  This follows from the fact that any constructible subset of a one dimensional variety is either open or closed (and hence a finite number of points), and if $D(s)$ is nonempty, it contains a dense open subset by Remark \ref{rem_denseopen}. If $D(s)$ is empty then we are done.
\end{proof}

\subsection{Cohen-Macaulay schemes}
 In the following we will only need the most basic facts about Cohen-Macaulay rings, a more thorough treatment can be found in \cite{BrunsHerzog-CM_rings}. Let $R$ be a commutative noetherian local ring with maximal ideal $\mathfrak{m}$, and let $M$ be a finitely generated $R$-module. A sequence of elements $(x_1,...,x_n) \subset \mathfrak{m}$ in a noetherian local ring $R$ is called a $M$-regular sequence if $x_i$ is a non-zero divisor on $M/(x_1,...,x_{i-1})M$ for all $i \leq n$ and $M/(x_1,...,x_n)M \neq 0$. The length of any maximal regular sequence is constant and is called the depth of $M$, denoted by $\operatorname{depth}M$. The depth should be thought of as an algebraic notion of dimension; in general the depth and dimension do not agree, but they are be related via $\operatorname{depth} M \leq \dim M$, where the dimension of a module is the dimension of its support as a sheaf on $\spec R$. Further, if $M$ has finite projective dimension over $R$, there is a close relationship between the depth of $R$ (as a module over itself) and the projective dimension of of $M$; $$\operatorname{proj.dim.} M + \operatorname{depth} M = \operatorname{depth} R,$$ known as the Auslander-Buchsbaum formula \cite[Theorem 1.3.3]{BrunsHerzog-CM_rings}.  Finally, a module is said to be a Cohen-Macaulay module if $\operatorname{depth} M = \dim M$, and a ring is Cohen-Macaulay if it is so as a module over itself. We can globalize this, namely a locally noetherian scheme $X$ is said to be a Cohen-Macaulay scheme if $\O_{X,x}$ is a Cohen-Macaulay ring for all $x\in X$. Our primary example is when $R$ is a reduced local ring of dimension one, where such rings are automatically Cohen-Macaulay (this is \cite[Exercise 2.1.20(a)]{BrunsHerzog-CM_rings}). Since the definition of a Cohen-Macaulay scheme is local, it follows that any reduced curve is Cohen-Macaulay. Cohen-Macaulay varieties can be singular, and any regular or Gorenstein variety is automatically Cohen-Macaulay.

 Suppose we are given a noetherian local ring $R$ which is Cohen-Macaulay. The ideal generated by a maximal regular sequence $\mathbf{x} = (x_1,...,x_n)$, $n =\dim R$, satisfies $\dim R/ (x_1,...,x_n) =0$ and the associated Koszul resolution $K(\mathbf{x})$ \cite[Section 1.6]{BrunsHerzog-CM_rings} gives a finite free resolution of the quotient. Further, in $\D^b(R)$, the bounded derived category of finitely-generated $R$-modules, the Koszul resolution satisfies $K(\mathbf{x})^\vee \cong K(\mathbf{x})[-\dim R]$ \cite[Proposition 1.6.10]{BrunsHerzog-CM_rings}. 
 
Now we give a definition orginally given in \cite[Section 1.3]{Ruiperez-Fourier-Mukai_transforms_for_Gorenstein_schemes}. 

\begin{definition}
  Let $X$ be a Cohen-Macaulay scheme. Given any closed point $x\in X$, the local ring $\O_{X,x}$ is a Cohen-Macaulay local ring, and hence has a regular sequence of length $\dim X$. A Koszul zero-cycle at $x$ is the closed subscheme $Z_x$ defined by the quotient of $\O_{X,x}$ by the ideal $\mathcal{I}_x$ generated by this regular sequence. 
\end{definition}
Clearly the structure sheaf $\O_{Z_x}$ of a Koszul zero-cycle $Z_x$ is a coherent sheaf supported only at the closed point $x\in X$, and is an object of $\perf X$ via the Koszul resolution. By the discussion above in the local case, we have that $\R\shom_{\O_X} (\O_{Z_x},\O_X) := \O_{Z_x}^\vee \cong \O_{Z_x}[-\dim X]$ in $\D^b(X)$.

Since we mainly have singular varieties in mind for applications, it is worth noting that the Koszul zero-cycles are a suitable replacement for the structure sheaves of closed points, for the main reason that they form a spanning set for both $\perf X$ and $\D^b(X)$. Recall that a collection of objects $\Omega$ in a triangulated category $\mathcal{T}$ is said to be spanning if for all $T\in \mathcal{T}$, $\hom(T,A[i])=0$ for all $i \in \Z$ and $A \in \Omega$ implies $T=0$, and $\hom(A,T[i])=0$ for all $i \in \Z$ and $A \in \Omega$ implies $T=0$. 

Before we prove that these objects are a spanning set, we include the following elementary lemma for the reader's convenience.

\begin{lemma}\label{lemma_local maps}
  Let $(R,\mathfrak{m},k)$ be a local noetherian ring and $M$ a finitely generated module. If $\mathfrak{m} \in \operatorname{supp} M$, $M$ admits a surjection $M \to k$. Further if $\{\mathfrak{m}\} = \operatorname{supp}M$, then $M$ also admits an injection $k \to M$.
\end{lemma}
\begin{proof}
  If $M$ is finitely generated and supported at the maximal ideal, then Nakayama's lemma implies that $M/\mathfrak{m}M \neq 0$, and further is a finite dimensional $k$-vector space. Choosing any further projection to a one-dimensional subspace, we have a surjection $M \to k$. 

  Consider a prime ideal $\mathfrak{p} \subset R$. Recall that $\mathfrak{p}$ is said to be associated to $M$ if there exists an element $m \in M$ such that $\mathfrak{p} = \operatorname{Ann}_R(m)$. Equivalently, if and only if $R/\mathfrak{p}$ is isomorphic to a submodule of $M$ . Further, since $R$ is noetherian, standard results in commutative algebra (\cite[\href{https://stacks.math.columbia.edu/tag/0586}{Lemma 0586},\href{https://stacks.math.columbia.edu/tag/0587}{Lemma 0587}]{stacks-project}) imply that the set of associated primes is always nonempty and is a subset of the support of $M$. Now if $\{\mathfrak{m}\} = \operatorname{supp}M$, it follows that $\mathfrak{m}$ is an associated prime to $M$, and hence there is a submodule of $M$ which is isomorphic to $R/m \cong k$. The inclusion of submodules $k \to M$ yields the claim.
\end{proof}
 
For the following lemma, recall that the support of a complex in $\D^b(X)$ is by definition the union of the supports of its cohomology sheaves. It is a closed subset of $X$, possibly non-reduced. We will always work with the induced reduced subscheme structure when necessary.

\begin{lemma}\label{lem_detect_support}
  Let $X$ be a Cohen-Macaulay quasi-projective variety, $x\in X$ be a closed point, $Z_x$ a Koszul zero cycle supported at $x$, and $Q \in \D^b(X)$. Then $x\in \operatorname{supp} Q$ if and only if there is some $i\in \Z$ such that $\hom(Q,\O_{Z_x}[i]) \neq 0$ if and only if there is some $j \in \Z$ such that $\hom(\O_{Z_x},Q[j]) \neq 0$.
\end{lemma}
 \begin{proof}
  Since $X$ is assumed Cohen-Macaulay, for any closed point there is a Koszul zero-cycle supported on it. Suppose a closed point $x\in X$ belongs to the support of $Q$ and denote the cohomology sheaves of $Q$ by $\mathcal{H}^i$. We consider the spectral sequence
  \begin{equation}\label{eqn_hyperext1}
     E^{p,q}_2 = \hom_{\D^b(X)}(\mathcal{H}^{-q},\O_{Z_{x}}[p]) \implies \hom_{\D^b(X)} (Q , \O_{Z_x}[p+q]).
  \end{equation}
  Since $x$ belongs to the support of $Q$, it must also belong to the support of at least one cohomology sheaf, say $\mathcal{H}^{q}$. Now using that structure sheaf of a Koszul zero-cycle is supported at a single closed point: $$\hom_{\O_{X,x}}(\mathcal{H}^{q}_{x},\O_{Z_{x}}) \cong \hom_{\O_X}(\mathcal{H}^{q},\O_{Z_{x}}).$$
   
  To show that these spaces are nonzero it is clearly enough to see that the left-hand side is nonzero.  In particular the stalk $\mathcal{H}^q_x$ is a finitely generated $\O_{X,x}$-module which is supported on the maximal ideal $\mathfrak{m}_x$ by assumption, and $\O_{Z_x}$ is a finitely generated $\O_{X,x}$ module whose support is precisely the maximal ideal. Hence by Lemma \ref{lemma_local maps} there is a nonzero morphism $$\mathcal{H}^q_x \to k(x) \to \O_{Z_x},$$ which shows that the morphism space is nonzero. 
   
  So now taking the maximal $q_0$ such that $\mathcal{H}^{q_0}_x \neq 0$ at $x$, we see that the term $E^{0,-q_0}_2$ survives to the $E_\infty$ page as there are no negative Ext groups between two sheaves, so the differential ending at $E^{0,-q_0}_2$ starts from zero, and the differential leaving $E^{0,-q_0}_2$ hits terms which are zero simply because we had chosen $q_0$ maximal, and all Ext groups between sheaves with disjoint support are trivial  (using the local-to-global spectral sequence). Hence this term survives, and since it is nonzero, we see $\hom_{\D^b(X)}(Q,\O_{Z_x}[-q_0]) \neq 0$.
 
  Conversely, suppose that $\hom_{\D^b(X)}(Q,\O_{Z_x}[m]) \neq 0$ for some $m \in \Z$. Suppose by way of contradiction that $x \notin \operatorname{supp} Q$, equivalently, $x \notin \operatorname{supp}\mathcal{H}^i$ for all $i \in \Z$. But since sheaves with disjoint support have no nontrivial Ext groups by the local-to-global spectral sequence, the spectral sequence \ref{eqn_hyperext1} yields a contradiction. This shows the first equivalence.
   
  To check the latter equivalence, we first observe as above that if the support of $Q$ and $\O_{Z_x}$ are disjoint, then the local-to-global spectral sequence shows that $\hom(\O_{Z_x},Q[j]) = 0$ for all $j\in \Z$. Conversely, assume that $x\in \operatorname{supp}Q$, and now we wish to use the spectral sequence $$E^{p,q}_2 = \hom(\O_{Z_x},\mathcal{H}^q[p]) \implies \hom(\O_{Z_x},E[p+q])$$ where we as above denote $\mathcal{H}^i(Q)$ by just $\mathcal{H}^i$. Since $Q \in \D^b(X)$, we have that $E^{p,q}_2 = 0$ for $|q|>>0$. Noting that $\O_{Z_x}$ is defined via a Koszul resolution, we see that $E^{p,q}_2 =0$ for $p \notin [0,\dim X]$.
 
  So now let $q_0$ be the maximal integer such that $\mathcal{H}^{q_0}$ is supported at $x\in X$. Note that for all $q>q_0$, since the supports are disjoint, $E^{p,q}_2 =0$. We claim that $E^{\dim X,q_0} \neq 0$. If so, it is clear that this term survives the spectral sequence and hence proves the proposition. To prove the claim, consider the following chain of isomorphisms:
   \begin{align*}
     E^{\dim X,q_0} & \cong \ext^{\dim X}_{\O_X} (\O_{Z_x},\mathcal{H}^{q_0}) \\
     & \cong \sext^{\dim X}_{\O_X} (\O_{Z_x},\mathcal{H}^{q_0}) \\
     & = \mathcal{H}^{\dim X}(\R \shom_{\O_X}(\O_{Z_x},\mathcal{H}^{q_0})) \\ 
     & \cong \mathcal{H}^{\dim X} (\O_{Z_x} \tens{\L} \mathcal{H}^{q_0}[-\dim X]) \\ 
     & \cong \mathcal{H}^{0} (\O_{Z_x} \tens{\L} \mathcal{H}^{q_0}) \\
     & \cong \O_{Z_x} \otimes \mathcal{H}^{q_0}
   \end{align*}
  where the second isomorphism follows from the local-to-global spectral sequence and the fourth follows as since $\O_{Z_x}$ is a Koszul zero cycle, $\O_{Z_x}^\vee \cong \O_{Z_x} [-\dim X]$. In the last line, since both are nonzero finitely generated modules over the local ring $\O_{X,x}$, this tensor product is nonzero, hence the claim.
\end{proof}

We also have the following result using very similar techniques. See \cite[Proposition A.91]{Bartocci-FM_in_geometry_and_physics}.
\begin{lemma}\label{lem_detect support2}
  Let $X$ be a quasi-projective variety, $x\in X$ be a closed point, and $E \in \D^b(X)$. Then $x\in \operatorname{supp} E$ if and only if there is some $i\in \Z$ such that $\hom(E,k(x)[i]) \neq 0$.
\end{lemma}

\begin{remark}\label{remark_indecomposable}
  As a final remark on the Koszul zero cycles, note that by their construction, they are local $k$-algebras, and hence indecomposable. That is, whenever $\O_{Z_x} \cong A \oplus B$ with $A,B \in \D^b(X)$, we must have that one of $A$ or $B$ is zero.
\end{remark}

\subsection{Rouquier functors}
All results in this section are contained in \cite[Section 4]{Rouquier-Dimensions_of_triangulated_categories} and \cite[Section 5]{Ballard-Derived_categories_of_singular_schemes_and_reconstruction}. 

Given a smooth projective variety $X$, the bounded derived category $\D^b(X)$ possesses a distinguished autoequivalence given by $$S_X(-) = (-) \tens{} \omega_X[\dim X] : \D^b(X) \to \D^b(X).$$ This functor captures the notion of Serre duality, and as such is called a Serre functor. More generally given a $k$-linear triangulated category $\T$, a Serre functor is a triangulated autoequivalence $S:\T \to \T$ for which there are natural isomorphisms $$\eta_{A,B}:\hom_{\T} (B,S(A)) \to \hom_{\T}(A,B)^*,$$ where ``$*$" indicates the vector space dual. 

When $X$ is singular, $\D^b(X)$ no longer possesses a Serre functor. When $X$ is Gorenstein however, we can instead work with $\perf X$.

\begin{proposition}[\cite{Ballard-Derived_categories_of_singular_schemes_and_reconstruction, SalasSalas-Reconstruction_from_the_derived_category}]
  If $X$ is a proper variety over a field $k$, the category $\perf X$ has a Serre functor if and only if $X$ is Gorenstein. In particular, the functor has the formula: $S_X(-) = (-) \tens{} \omega_X [\dim X]$.
\end{proposition}

The failure of $\perf X$ to have a Serre functor in the non-Gorenstein case is essentially a failure of existence, as the functors $\hom(A,-)$ are no longer represented by objects in $\perf X$. However $\hom(A,-)$ \emph{can} be represented by objects in $\D^b(X)$. This observation is due to Rouquier \cite[Section 4]{Rouquier-Dimensions_of_triangulated_categories}, where the following more general definition was given. 
\begin{definition}
  Let $\mathcal{S}$ and $\T$ be arbitrary $k$-linear triangulated categories, and $F:\mathcal{S} \to \T$ be a $k$-linear triangulated functor. A Rouquier functor associated to $F$ is a $k$-linear triangulated functor $R_F : \mathcal{S} \to \T$ for which there are natural isomorphisms $$\eta_{A,B}: \hom_{\T} (B,R_F(A)) \to \hom_{\T}(F(A),B)^*.$$
\end{definition}

Our main example is the case of projective schemes. For the following, see Proposition 7.47 and Remark 7.48 in \cite{Rouquier-Dimensions_of_triangulated_categories}, see also \cite[Example 5.12]{Ballard-Derived_categories_of_singular_schemes_and_reconstruction} for more details.

\begin{lemma}\label{lemma_Rouquier functor is dualizing complex}
  Let $f:X \to \spec k$ be a projective scheme over the field and $\iota:\perf X \hookrightarrow D(\qcoh X)$ the inclusion functor. Then the Rouquier functor $R_X := R_\iota$ exists and is isomorphic to $$(-) \tens{\L} f^! \O_{\spec k}.$$ Furthermore, it is a fully faithful functor mapping $\perf X$ into $\D^b(X)$.
\end{lemma}

In our situation, where the projective scheme $X$ is a Cohen-Macaulay variety, the Rouquier functor becomes $$R_X(-) = (-) \tens{\L} \omega_X [\dim X],$$ where $\omega_X$ is the dualizing sheaf in the sense of \cite{Hartshorne-AG}. It is well-known that for a Cohen-Macaulay variety, the dualizing sheaf is a divisorial sheaf, that is, reflexive (in particular, torsion-free) and generically rank 1. When the variety is Gorenstein, the dualizing sheaf is actually a line bundle (\cite[Proposition V.9.3]{Hartshorne-Residues_and_Duality}). 

%% file: sods.tex
\subsection{Semiorthogonal decompositions}
Recall the following definition.

\begin{definition}\label{def_admissible}
  Given a triangulated category $\T$, a triangulated subcategory $\A \subset \T$ is said to be (left) right admissible if the inclusion functor $i:\A \to \T$ admits a (left) right adjoint ($i^*$) $i^!$. We say simply that $\A$ is admissible if it is both left and right admissible.
\end{definition}
\noindent Now let us remind the reader of the notion of a semiorthogonal decomposition.

\begin{definition}\label{def_semi_dec}
  Let $\T$ be a triangulated category. A pair of strictly full (admissible) triangulated subcategories $\A, \B$ of $\T$ is a (admissible) semiorthogonal decomposition of $\T$ if the following conditions are satisfied:
  \begin{enumerate}
    \item $\A \subset \B^\perp$, and
    \item for any $T \in \T$, there is a triangle $$B \to T \to A \to B[1]$$ with $A \in \A$ and $B \in \B$.
  \end{enumerate} 
  We then write $\T = \langle \A,\B \rangle$.
\end{definition}

\begin{remark}
  Given a left (right) admissible subcategory $\A \subset \T$, we have a semiorthogonal decomposition  $\langle \A, \prescript{\perp}{}{\A}\rangle$ (respectively $\langle \A^\perp,\A \rangle$).  Conversely if $\T = \langle \A,\B\rangle$, then $\A$ ($\B$) is automatically left (right) admissible \cite{Bondal-Quivers}.
\end{remark}

\begin{remark}
  When $X$ is smooth, any semiorthogonal decomposition $\D^b(X) = \langle \A, \B \rangle$ is automatically admissible. However in the singular case, this is no longer true. Admissible semiorthogonal decompositions are significantly better behaved, as the following very useful lemma suggests.
\end{remark}

\begin{lemma}[\cite{Orlov-Triangulated_cats_of_singularities}, Lemma 1.10]\label{lemma-restrict-sod}
  Let $X$ be a projective variety and suppose that $\D^b(X) = \langle \A_1,...,\A_n\rangle$ is an admissible semiorthogonal decomposition. Then $$\perf X = \langle  \A_1\cap \perf X,...,\A_n \cap \perf X\rangle$$ is an admissible semiorthogonal decomposition for $\perf X$. 
\end{lemma}

\begin{remark}
  If, in Definition \ref{def_semi_dec}, we also required the dual condition that $\B \subset \prescript{\perp}{}{\A}$, so that there are no morphisms from objects of $\A$ to objects of $\B$, then the corresponding notion is that of an orthogonal decomposition. It is well known that a connected scheme admits no nontrivial orthogonal decompositions of either $\perf X$ or $\D^b(X)$.
\end{remark}

\section{Main results}

\begin{definition}
  Let $\T$ be a triangulated category. We say that $\T$ is (weakly) indecomposable if any (admissible) semiorthogonal decomposition is trivial. That is, whenever $\T = \langle \A,\B \rangle$, we have that either $\A=0$ or $\B=0$.
\end{definition}

\noindent The original version of the following theorem can be found \cite[Theorem 3.1]{KawataniOkawa-Nonexistence}, where it is proven for smooth DM stacks with at least one non-stacky point. Here we prove an extension to Cohen-Macaulay projective varieties.

\begin{theorem}\label{theorem-main}
  Let $X$ be a Cohen-Macaulay projective variety with dualizing sheaf $\omega_X$. Suppose that $\perf X = \langle \A, \B \rangle$. Then:
  \begin{enumerate}
    \item for any closed point $x \in X \setminus \operatorname{Bs}|\omega_X|$, and all Kozsul zero-cycles $Z_x$ at $x$, exactly one of the following holds:
    \begin{enumerate}
      \item $\O_{Z_x} \in \A$, or
      \item $\O_{Z_x} \in \B$.
    \end{enumerate}
    \item When (1a) (resp. (1b)) is satisfied, then the support of any object in $\B$ (resp. $\A$) is contained in $\operatorname{Bs}|\omega_X|$.
  \end{enumerate}
\end{theorem}
The following is nearly verbatim from \cite[Theorem 3.1]{KawataniOkawa-Nonexistence}, but we include the proof with our modifications for the convenience of the reader. 
\begin{proof}
  Let $x\in X \setminus \operatorname{Bs}|\omega_X|$ be a closed point and, using that $X$ is Cohen-Macaulay, a Koszul zero cycle $Z_x$ supported at $x$. Suppose that $\perf X = \langle \A, \B \rangle$, and consider the triangle $$B \to \O_{Z_x} \to A \overset{f}{\to} B[1]$$ which is obtained from the semiorthogonal decomposition. If $f=0$, then $\O_{Z_x} \cong A \oplus B$, but this implies that one of $A$ or $B$ is zero (see Remark \ref{remark_indecomposable}), and hence $\O_{Z_x} \in \B$ or $\A$ respectively, so we may assume that $f \neq 0$. 
  
  Now choose a global section $s\in H^0(X,\omega_X) \cong \hom(\O_X,\omega_X)$ such that $s(x) \neq 0$ and $x\in D(s)^\circ$, the interior of $D(s)$. We may choose such a global section as $x \notin \operatorname{Bs}|\omega_X|$, and by definition, the base locus of $\omega_X$ is the complement of the union of the interiors of the sets $D(s)$ (and in particular, is a closed subset of $X$). 
  
  We now consider the composition $$A \otimes \O_X \overset{f \otimes \mathbbm{1}}{\to} B[1] \otimes \O_X \overset{\mathbbm{1} \otimes s}{\to} B[1] \otimes\omega_X.$$ Using the Rouqiuer functor, the composition lies in $$\hom_{\D^b(X)}(A,B\otimes \omega_X[1]) \cong \hom_{\perf X}(B,A[\dim X -1])^* =0,$$ where the latter space is zero by the assumption that $\perf X = \langle \A ,\B \rangle$. Now let $U \subset D(s)^\circ$ be an open neighborhood of $x$, and set $j :U  \hookrightarrow X$ to be the open immersion. We now examine $j^* (f \otimes s) : = (f \otimes s)|_U = f|_U \otimes s|_U$ in exactly the same manner as \cite[Theorem 3.1]{KawataniOkawa-Nonexistence}. 
  
  We know that this composition must be identically zero, and further we know that $j^* s := s|_U \neq 0$ by our choice of $U$. So we see that $f|_U = 0$; then since $j^*$ is triangulated, the restriction of the triangle $$B|_U \to \O_{Z_x}\to A|_U \overset{0}{\to} B[1]|_U$$ implies that $\O_{Z_x} \cong A|_U \oplus B|_U$. As before, one of the two summands must be trivial. Suppose that $A|_U \cong 0$. Then $\O_{Z_x} \cong B|_U$, and since the support of $A$ and $\O_{Z_x}$ is disjoint, it follows that the morphism $\O_{Z_x} \to A$ is zero, and hence $B \cong \O_{Z_x} \oplus A[-1]$. If $A[-1] \neq 0$, then there must exist a nonzero projection $B \to A[-1]$, contradicting that the decomposition was semiorthogonal. Hence $A[-1] \cong 0$, and $\O_{Z_x} \in \B$. If instead $B|_U \cong 0$, we similarly obtain $\O_{Z_x} \in \A$.   

  For the second claim, we also proceed in the same manner as in \cite[Theorem 3.1]{KawataniOkawa-Nonexistence}, but we use Lemma \ref{lem_detect_support} in place of the corresponding \cite[Lemma 2.8]{KawataniOkawa-Nonexistence}.  Notice that if $\O_{Z_x} \in \A$ (or $\B$) for some closed point $x \in X$, then any object in $\B$ (or $\A$) must have proper support, by Lemma \ref{lem_detect_support}. So now suppose that both $\A$ and $\B$ contain Koszul zero cycles, and consider the canonical filtration of the structure sheaf $$B \to \O_X \to A \to B[1].$$ It then follows that the support of $\O_X$ is the union of the supports of $A$ and $B$. But since both $\A$ and $\B$ contain Koszul zero cycles the supports of $A$ and $B$ are proper, which is a contradiction since we have assumed $X$ is irreducible. Hence only one of the two subcategories can contain any of the Koszul zero cycles. 
\end{proof}

\begin{corollary}\label{corollary-main}
  Suppose that $X$ is a Cohen-Macaulay projective variety such that every connected component of $\operatorname{Bs}|\omega_X|$ is contained in an open subset on which $\omega_X$ is trivial. Then $\perf X$ is indecomposable. If $\omega_X$ is globally generated, then $\D^b(X)$ is weakly indecomposable.
\end{corollary}
\begin{proof}
  Suppose that every connected component of $\operatorname{Bs}|\omega_X|$ is contained in an open subset on which $\omega_X$ is trivial. Then by Theorem \ref{theorem-main}, given any nontrivial semiorthogonal decomposition $\perf X = \langle \A, \B \rangle$, one of the components must satisfy the condition that all of its objects must be supported in $\operatorname{Bs}|\omega_X|$. Without loss of generality we can assume that this component is $\B$. Then for any $B \in \B$, since its support is contained in an open subset where $\omega_X$ is trivial, $R_X(B) \cong B$. Hence for any object $A \in \A$, and any $n \in \Z$, $$\hom (B,A[n]) \cong \hom(A,B[n+\dim X])^* =0$$ by semiorthogonality, and so the decomposition $\perf X = \langle \A,\B \rangle$ is actually an orthogonal decomposition. But this contradicts that $X$ is connected (as we have assumed it is integral). 
  
  Now suppose that $\operatorname{Bs}|\omega_X|$ is empty. Consider an admissible semiorthogonal decomposition $\D^b(X) = \langle \A, \B \rangle$. Then this induces a semiorthogonal decomposition of $\perf X$ by Lemma \ref{lemma-restrict-sod}, and by Theorem \ref{theorem-main}, one of the components, say $\tilde \A$ for concreteness, must contain the spanning set of Koszul zero-cycles $\{\O_{Z_x}\,|\, x\in X \text{ closed}\}$. It then follows that $\A$ also contains the spanning set of Koszul zero cycles since $\tilde \A = \perf X \cap \A$, and hence both $\B$ and $\tilde \B$ are zero. In particular, the admissible decomposition $\D^b(X) = \langle \A, \B \rangle$ must be trivial.
\end{proof}

These results alone are interesting, but due to the nature of the dualizing sheaf, the base locus $\operatorname{Bs}|\omega_X|$ can be difficult to analyze in many situations if $\omega_X$ is not globally generated. The next simplest situation is when $\operatorname{Bs}|\omega_X|$ is a finite number of points. To analyze this, we formulate two more results of possibly independent interest.

\begin{lemma} \label{prop_main}
  Let $X$ be a projective variety, and let $Y \subset X$ be a finite collection of isolated closed points. Suppose that $\A \subset \D^b(X)$ is a left admissible subcategory such that the support of every object in $\A$ is contained in $Y$. Then $\A=0$.
\end{lemma}
\begin{proof}
  Suppose that the support of every object in $\A$ is contained in $Y$. Since $\A$ is left admissible, there is a semiorthogonal decomposition $\langle \A ,\prescript{\perp}{}{\A} \rangle = \D^b(X)$ (where $\prescript{\perp}{}{\A}$ is only right admissible in general). Assume by contradiction that $\A \neq 0$, and choose some nonzero object $A\in \A$. Using that the support is isolated, $A \cong \bigoplus_{y_i \in Y} A_i$, where the support of $A_i$ is exactly the closed point $y_i \in Y$. Now choose $B\in \prescript{\perp}{}{\A}$ so that its support is all of $X$. We can guarantee that such an object exists by looking at the distinguished filtration of $\O_X$ induced by $\langle \A ,\prescript{\perp}{}{\A} \rangle$ and the long exact sequence of cohomology sheaves.

  By Lemma \ref{lem_detect support2}, we then see that there are integers $j_i \in \Z$ such that $$\hom_{\D^b(X)}(B[j_i],k(y_i)) \neq 0.$$ Similarly, if $l_i$ is the minimal $l_i \in \Z$ such that $\mathcal{H}^{l_i}(A_i) \neq 0$ (recall the support of $A_i$ is exactly $y_i$), we have, by Lemma \ref{lemma_local maps}, a nonzero morphism $k(y_i) \to \mathcal{H}^{l_i}(A_i)$ realizing an isomorphism between the residue field and a submodule. The composite morphism $B[j_i] \to \mathcal{H}^{l_i}(A_i)$ is nonzero. Now composing with the morphisms $\mathcal{H}^{l_i}(A_i) \to A_i[l_i] \to A[l_i]$, where the first morphism uses that $l_i$ was minimal, we can see that the composition $B[j_i] \to A[l_i]$ is a nonzero map between two objects of $\prescript{\perp}{}{\A}$ and $\A$, which contradicts semiorthogonality. Hence we must have $\A=0$. 
\end{proof}

\begin{lemma}\label{lem_perf support}
  Let $X$ be a projective variety and $F,G \in \perf X$ such that $\operatorname{supp}F \cap \operatorname{supp}G$ contains an isolated point. Then $\R\hom(F,G) \neq 0$.
\end{lemma}
\begin{proof}
  This follows as $\R \hom_{\perf X}(F,G) \cong \R \Gamma(X,F^* \otimes G)$. Since the support of $F^* \otimes G$ contains an isolated point, this complex decomposes into a nontrivial direct sum, with the support of one summand exactly the isolated point. Any object with zero-dimensional support has a nonvanishing (hyper)cohomology class given by a nonzero global section of the lowest degree cohomology sheaf, so we are done.
\end{proof}

\begin{corollary}\label{cor_point support sod}
  Let $X$ be a projective variety and assume that $\perf X =\langle \A,\B \rangle$ or $\D^b(X)=\langle \A,\B \rangle$, is a semiorthogonal decomposition, respectively, admissible semiorthogonal decomposition. Suppose that $Y \subset X$ be a finite collection of isolated closed points and that the support of every object in $\A$ or $\B$ is contained in $Y$. Then the semiorthogonal decomposition is trivial.
\end{corollary}
\begin{proof}
   For $\D^b(X)$, since we have assumed that the decomposition is admissible, this follows from Lemma \ref{prop_main}. Now for $\perf X$, this is handled by Lemma \ref{lem_perf support}.
\end{proof}

As we hinted at earlier, these results allow us to also resolve the question of indecomposability when $\operatorname{Bs}|\omega_X|$ is finite.

\begin{corollary}\label{corollary_isolated}
  Let $X$ be a Cohen-Macaulay projective variety such that $\operatorname{Bs}|\omega_X|$ is a finite number of isolated closed points. Then $\perf X$ is indecomposable and $\D^b(X)$ is weakly indecomposable.
\end{corollary}
\begin{proof}
  Suppose we have a decomposition $\perf X = \langle \tilde\A,\tilde\B \rangle$. By Theorem \ref{theorem-main}, either $\tilde \A$ or $\tilde \B$ contains all Koszul zero-cycles $\O_{Z_x}$ for $x \notin \operatorname{Bs}|\omega_X|$. Without loss of generality we can assume its $\tilde\A$.  Hence $\tilde \B$ satisfies the hypotheses in Corollary \ref{cor_point support sod}, so the decomposition of $\perf X$ is trivial. We see that this category is indecomposable. 
  
  Now suppose we have an admissible decomposition $\D^b(X) = \langle \A,\B\rangle$. This, by Lemma \ref{lemma-restrict-sod}, induces $\perf X = \langle \tilde\A,\tilde\B \rangle$, where $\tilde \A = \perf X \cap \A$ and $\tilde \B = \perf X \cap \B$. The above paragraph implies (again, without loss of generality) that $\A \subset \D^b(X)$ also contains all Koszul zero-cycles $\O_{Z_x}$ for $x \notin \operatorname{Bs}|\omega_X|$. So $\B$ satisfies the hypotheses in Corollary \ref{cor_point support sod}, hence $\B =0$. Since this decomposition was arbitrary, we are done.
\end{proof}

The following corollary relates these observations to curves. Note that it is known that smooth and Gorenstein curves of arithmetic genus larger then one have empty canonical base locus (\cite[Theorem 1.6]{Hartshorne-Generalized_divisors} and \cite[Theorem D]{Catanese-Pluricanonical_gorenstein_curves}). So Corollary \ref{corollary-main} resolves this question for Gorenstein curves. In the Cohen-Macaulay case however, we do not know if the canonical base locus is empty, but we can still conclude the same result.

\begin{corollary}\label{corollary-curves}
  Let $C$ be a projective curve of arithmetic genus at least one. Then $\perf C$ is indecomposable and $\D^b(C)$ is weakly indecomposable.  
\end{corollary}
\begin{proof}
  First note that any projective curve is Cohen-Macaulay. By Corollary \ref{corollary_isolated}, it is enough to see that the base locus of the dualizing sheaf $\omega_C$ consists of isolated closed points.

  Let $p \in C$ be a regular point; given the exact sequence $$0 \to \omega_C(-p) \to \omega_C \to k(p) \to 0,$$ we see that $p$ is a base point of $\omega_C$ if and only if the map $$H^1(\omega_C(-p)) \to H^1(\omega_C) \to 0$$ is not an isomorphism, equivalently $h^1(\omega_C(-p)) \geq 2$. We claim that in fact $h^1(\omega_C(-p)) = 2$.  Indeed, since $p$ is regular, Riemann-Roch informs us that $h^0(p) - h^1(p) = 2-p_a(C)$. Further, from the long exact sequence of cohomology arising from the short exact sequence above, we have $h^1(p) = h^0(\omega_C(-p)) \leq p_a(C)$. Rearranging, we have $$h^0(p) - 2 = h^0(\omega_C(-p)) - p_a(C),$$ where the left hand side is nonnegative, and the right hand side is nonpositive, hence they must both be equal, and $h^0(p) = 2$.
  
  Thus $p$ is a base point of $\omega_C$ if and only if $h^0(C,\O_C(p)) = 2$. However this latter equality holds only if $p_a(C) =0$. This shows that if $C_{reg}$ is the locus of regular points, $C_{reg} \cap \operatorname{Bs}|\omega_C|= \emptyset$, in particular, $\operatorname{Bs}|\omega_C|$ is isolated.
\end{proof}